\documentclass[10pt]{amsart}
\usepackage{amsmath, latexsym, amsthm, amsfonts,bm,amssymb} 
\usepackage{graphicx} 
\usepackage{epsfig}
\usepackage{varioref}
\usepackage{eucal}
\usepackage{subfloat}
\usepackage{subfig}
\usepackage{appendix}
\usepackage{natbib} 
\usepackage{url} 

\bibpunct{(}{)}{;}{a}{}{,} 

\theoremstyle{plain}
\newtheorem{thm}{Theorem}

\newtheorem{lemma}{Lemma}

\newtheorem{assump}{Assumption}

\theoremstyle{remark}
\newtheorem{rem}{Remark}

\def\convd{\stackrel{\mathcal{D}}{\rightarrow}}
\def\convp{\stackrel{\mathbb{P}}{\rightarrow}}

\def\ex{{\rm {\mathbb E\,}}}
\def\var{{\rm {\mathbb Var\,}}}
\def\cov{{\rm {\mathbb Cov\,}}}

\newcommand{\pp}{\mathbb{P}}
\newcommand{\ee}{\mathbb{E}}

\newcommand{\esssup}{\mathrm{ess\, sup}}

\allowdisplaybreaks

\begin{document}

\title[Parametric inference for SDEs]{Parametric inference for stochastic differential equations: a smooth and match approach}

\author{Shota Gugushvili}
\address{Mathematical Institute\\
Leiden University\\
P.O. Box 9512\\
2300 RA Leiden\\
The Netherlands}
\email{shota.gugushvili@math.leidenuniv.nl}
\author{Peter Spreij}
\address{Korteweg-de Vries Institute for Mathematics\\
University of Amsterdam\\
P.O.\ Box 94248\\
1090 GE Amsterdam\\
The Netherlands}
\email{spreij@uva.nl}



\begin{abstract}

We study the problem of parameter estimation for a univariate discretely observed ergodic diffusion process given as a solution to a stochastic differential equation. The estimation procedure we propose consists of two steps. In the first step, which is referred to as a smoothing step, we smooth the data and construct a nonparametric estimator of the invariant density of the process. In the second step, which is referred to as a matching step, we exploit a characterisation of the invariant density as a solution of a certain ordinary differential equation, replace the invariant density in this equation by its nonparametric estimator from the smoothing step in order to arrive at an intuitively appealing criterion function, and next define our estimator of the parameter of interest as a minimiser of this criterion function. Our main results show that under suitable conditions our estimator is $\sqrt{n}$-consistent, and even asymptotically normal. We also discuss a way of improving its asymptotic performance through a one-step Newton-Raphson type procedure and present results of a small scale simulation study.

\end{abstract}

\date{\today}

\maketitle

\section{Introduction}
\label{intro}
Stochastic differential equations play an important role in modelling various phenomena arising in fields as diverse as finance, physics, chemistry, engineering, biology, neuroscience and others, see e.g.\ \cite{allen}, \cite{hindriks}, \cite{musiela} and \cite{wong}.\ These equations usually depend on parameters, which are often unknown. On the other hand knowledge of these parameters is critical for the study of the process at hand and hence their estimation based on the observational data on the process under study is of great importance in practical applications. A formal setup that we consider in this paper is as follows: let $(\Omega, \mathcal{F}, \mathbb{P})$ be a  probability space. Consider a Brownian motion $W=(W_t)_{t\geq 0}$ and a random variable $\xi$ independent of $W$ that are defined on $(\Omega, \mathcal{F}, \mathbb{P})$ and let $\mathfrak{F}=(\mathfrak{F}_t)_{t \geq 0}$ be the augmented filtration generated by $\xi$ and $W.$ Consider a stochastic differential equation driven by $W,$
\begin{equation}
\label{sde}
\begin{cases}
dX_t=\mu(X_t;\theta)dt+\sigma(X_t;\theta)dW_t,\\
X_0=\xi,
\end{cases}
\end{equation}
where $\theta\in\Theta\subset\mathbb{R}$ is an unknown parameter and $X_0=\xi$ defines the initial condition. Assume that there exists a unique strong solution to \eqref{sde} on $(\Omega, \mathcal{F}, \mathbb{P})$ with respect to the Brownian motion $W$ and initial condition $\xi.$ Let $\theta_0$ denote the true parameter value. Furthermore, let $X$ be ergodic with invariant density $\pi(\cdot;\theta_0)$ and let $\xi\sim \pi(\cdot;\theta_0).$ The solution $X$ is thus a strictly stationary process. Given a discrete time sample $X_0,X_{\Delta},X_{2\Delta},\ldots,X_{n\Delta}$ from the process $X,$ our goal is to estimate the parameter $\theta_0.$ Hence here we consider a parametric inference problem for a stochastic differential equation. There is also a rich body of literature on nonparametric inference for stochastic differential equations, see e.g.\ \cite{comte}, \cite{reiss} and \cite{jacod} and references therein. A general reference on statistical inference for ergodic diffusion processes is \cite{kutoyants}.

A natural approach to estimation of $\theta_0$ is the maximum likelihood method. Assume that the transition density $p(\Delta,x,y;\theta)$ of $X$ exists. Then the likelihood function associated with observations $X_{0},X_{\Delta},\ldots,X_{n\Delta}$ can be written as
\begin{equation*}
p(X_0,X_{\Delta},X_{2\Delta},\ldots,X_{n\Delta};\theta)=\pi(X_{0};\theta)\prod_{j=0}^{n-1}p(\Delta,X_{j\Delta},X_{(j+1)\Delta};\theta),
\end{equation*}
and the maximum likelihood estimator can be computed by maximising the right-hand side of this expression over $\theta,$ provided both the invariant density and the transition density are known explicitly. Unfortunately, for many realistic and practically useful models transition densities are not available in explicit form, which makes exact computation of the maximum likelihood estimator impossible. In those cases when the likelihood cannot be evaluated analytically, a number of alternative estimators have been proposed in the literature, which try to emulate the maximum likelihood method and rely upon some approximation of the likelihood, whence their name, the approximate maximum likelihood estimators, derives. For an overview and relevant references see e.g.\ Section 5 in \cite{h_sorensen}. Although successful in a significant number of examples, these methods typically suffer from a considerable computational burden, see a brief discussion on pp.\ 350--351 in \cite{h_sorensen}.  We also remark that in general in statistical problems if the likelihood is a nonlinear function of the parameter of interest, computation of the maximum likelihood estimator is often far from straightforward, see e.g.\ \cite{barnett}. Returning then to diffusion processes, even if the transition densities are explicitly known, they still might be highly nonlinear functions of the parameter $\theta,$ which might render maximisation of the log-likelihood a difficult task. This is in particluar true for the Cox-Ingersoll-Ross (CIR) process (see e.g.\ pp.\ 356--358 in \cite{musiela} for more information on the CIR process), where the transition densities are noncentral chi-square densities, already numerical evaluation of which, saying nothing about the optimisation process itself, is a nontrivial task, see \cite{dyrting}.

A popular alternative to approximate maximum likelihood methods is furnished by Z-estimators, which are defined as zeroes in $\theta$ of estimating equations
\begin{equation*}
F_n(X_0,X_{\Delta},\ldots,X_{n\Delta};\theta)=0
\end{equation*}
for some given functions $F_n.$   For a general introduction to Z-estimators see e.g.\ Chapter 5 in \cite{vdvaart}. Z-estimators are often faster to compute than approximate maximum likelihood estimators, but the question of the choice of the estimating equations is a subtle one with no readily available recipes in many cases. For instance, the existing methods at times yield choices of $F_n$ that might give rise to numerical problems or that are infeasible in practice, see remarks on pp.\ 343--344 in \cite{h_sorensen}. For additional information on this approach to parameter estimation for diffusion processes and references see \cite{m_sorensen}, \cite{jacobsen}, \cite{kessler} and Section 4 in \cite{h_sorensen}.

In the present work we study an approach alternative to the ones surveyed above. In particular, we will use a characterisation of the invariant density $\pi(\cdot;\cdot)$ of \eqref{sde} as a solution of the ordinary differential equation (here a prime denotes a derivative with respect to $x$)
\begin{equation}
\label{ode}
\mu(x;\theta){\pi}(x;\theta) - \frac{1}{2}\left[ \sigma^2(x;\theta) {\pi}(x;\theta) \right]^{\prime}=0,
\end{equation}
to motivate an estimator $\hat{\theta}_n$ of $\theta_0$ defined as
\begin{equation}
\label{est_theta}
\hat{\theta}_n = \operatorname{argmin}_{\theta\in\Theta} R_n(\theta),
\end{equation}
where
\begin{equation}
\label{crf}
R_n(\theta)=\int_{\mathbb{R}}\left( \mu(x;\theta)\hat{\pi}(x) - \frac{1}{2}\left[ \sigma^2(x;\theta) \hat{\pi}(x) \right]^{\prime} \right)^2 w(x)dx.
\end{equation}
Here $w(\cdot)$ is a weight function chosen beforehand and $\hat{\pi}(\cdot)$ is a nonparametric estimator of ${\pi}(\cdot;\theta_0).$ In particular, in the latter capacity we will use a kernel density estimator. The intuition for $\hat{\theta}_n$ is that for $\hat{\pi}(\cdot)$ that is `close' to ${\pi}(\cdot;\theta_0),$ in view of \eqref{ode} the same must be true for $\hat{\theta}_n$ and $\theta_0.$ 

The estimator $\hat{\theta}_n$ will be called a smooth and match estimator. Its name reflects the fact that it is obtained through a two-step procedure: in the first step, which is referred to as a smoothing step, the data $Z_0,Z_1,\ldots,Z_n$ are smoothed in order to obtain a nonparametric estimator $\hat{\pi}(\cdot)$ of the stationary density $\pi(\cdot;\theta_0).$ In the second step, which is referred to as a matching step, a characterisation of $\pi(\cdot;\theta_0)$ as a solution of \eqref{ode} is used and an estimator of $\theta_0$ is obtained in such a way that the left-hand side of \eqref{ode} with $\pi(\cdot;\theta_0)$ replaced by $\hat{\pi}(\cdot)$  approximately matches zero. The construction of the estimator $\hat{\theta}_n$ is motivated by a similar construction used in parameter estimation problems for ordinary differential equations, see \cite{gugu} for additional information and references. Approaches to parameter estimation for stochastic differential equations that are close in spirit to the one considered in the present work, in that they rely on matching a parametric function to its nonparametric estimator, are studied in \cite{ait2}, \cite{bandi}, \cite{kristensen} and \cite{h_soren}. We remark that our approach differs from the approaches in these papers either by the type of asymptotics or by the criterion function.

The estimator $\hat{\theta}_n$ is especially straightforward to compute when the drift coefficient $\mu(\cdot;\cdot)$ is linear in the components of the parameter $\theta,$ see Remark \ref{rem_estex} below. Obviously, ease of computation cannot be a sole justification for the use of any  particular estimator and hence in order to provide more motivation for the use of our estimator in the present work we will study its asymptotic properties. Since the estimator $\hat{\theta}_n$ is ultimately motivated by a characterisation of the marginal density of $X,$ in the most general setting when both the drift and the dispersion coefficients  in \eqref{sde} depend on the parameter $\theta,$ the full parameter vector $\theta$ will typically be impossible to estimate due to identifiability problems. We hence have to specialise to some particular case, and we do this for the case when the dispersion coefficient $\sigma(\cdot;\theta)$ does not depend on $\theta$ and is a known function $\sigma(\cdot).$ Thus the stochastic differential equation underlying our model is
\begin{equation}
\label{sde0}
\begin{cases}
dX_t=\mu(X_t;\theta)dt+\sigma(X_t)dW_t,\\
X_0=\xi.
\end{cases}
\end{equation}

The structure of the paper is as follows: in Section \ref{assumptions} for the reader's convenience we list together the assumptions on our model. Detailed remarks on these assumptions are given in Section \ref{remarks}. When reading the paper, a reader can either browse through Section \ref{assumptions} or refer to it as need arises in the subsequent sections. A reader who finds the assumptions in Section \ref{assumptions} believable can skip Section \ref{remarks} at first reading. In Sections \ref{results} and \ref{normality} we state the main results of the paper, namely $\sqrt{n}$-consistency and asymptotic normality of $\hat{\theta}_n.$ In Section \ref{onestep} we discuss a further asymptotic improvement of the estimator $\hat{\theta}_n$ through a Newton-Raphson type procedure. Results of a small simulation study are presented in Section \ref{simulations}. Section \ref{proofs} contains proofs of the results from Sections \ref{results} and \ref{normality}. Finally, Appendix \ref{appendix} contains several technical lemmas used in the proofs of the results from Sections \ref{results} and \ref{normality}.

We remark that in the present work we do not strive for maximal generality. Rather, our goal is to explore asymptotic properties of an intuitively appealing estimator of $\theta_0,$ and to show that this estimator leads to reasonable results in a number of examples.

Throughout the paper we use the following notation for derivatives: a dot denotes a derivative of an arbitrary function $q(x;\theta)$ with respect to $\theta,$ while a prime denotes its derivative with respect to $x.$ We also define the strong mixing coefficient $\alpha_{\Delta}(k)$ as
\begin{equation*}
\alpha_{\Delta}(k)=\sup_{m \geq 0}\sup_{A\in \mathfrak{F}_{\leq m},B\in \mathfrak{F}_{\geq m+k} } |\mathbb{P}(AB)-\mathbb{P}(A)\mathbb{P}(B)|,
\end{equation*}
where $\mathfrak{F}_{\leq m}=\sigma(Z_j,j\leq m)$ and $\mathfrak{F}_{\geq m}=\sigma(Z_j,j\geq m)$ for $m\in\mathbb{N}\cup\{0\}.$ Here $Z_j=X_{j \Delta}$ for $j\in\mathbb{N}\cup\{0\}.$ We call the sequence $Z_j$ $\alpha$-mixing (or strongly mixing), if $\alpha_{\Delta}(k)\rightarrow 0$ as $k\rightarrow\infty.$ 
When comparing two sequences $\{ a_n \}$ and $\{ b_n \}$ of real numbers, we will use the notation $a_n\lesssim b_n$ to denote the fact that $\exists C>0,$ such that for $\forall n\in\mathbb{N}$ the inequality $a_n\leq C b_n$ holds. A similar convention will be used for $a_n \gtrsim b_n.$ The notation $a_n \asymp b_n$ will denote the fact that the sequences $\{ a_n \}$ and $\{ b_n \}$ are asymptotically of the same order.

\section{Assumptions}
\label{assumptions}

In this section we list the assumptions under which the theoretical results of the paper are proved. 

\begin{assump}
\label{cnd_theta}
The parameter space $\Theta$ is a compact subset of $\mathbb{R}$: $\Theta=[a,b]$ for $a<b.$
\end{assump}

\begin{assump}
\label{assump_sol}
The drift coefficient $\mu(\cdot;\theta)$ is known up to the parameter $\theta$ and the dispersion coefficient $\sigma(\cdot)$ is a known function. Furthermore, there exists a unique strong solution $X=(X_t)_{t\geq 0}$ to \eqref{sde0} on $(\Omega, \mathcal{F}, \mathbb{P})$ with respect to the Brownian motion $W$ and initial condition $\xi.$ It is a homogeneous Markov process with transition density $p(t,x,y;\theta).$ Moreover, this solution is ergodic with bounded invariant density $\pi(\cdot;\cdot)$ that has a bounded, continuous and integrable derivative $\pi^{\prime}(\cdot;\cdot),$ and for $\xi\sim\pi(\cdot;\theta)$ the solution $X$ is a strictly stationary process. Also, $\dot{\pi}(\cdot,\cdot)$ exists. Finally, for all $\theta\in\Theta$ it holds that the support of $\pi(\cdot;\cdot),$ i.e.\ the state space of $X,$ equals $\mathbb{R}.$
\end{assump}

\begin{assump}
\label{cnd_mixing}
A sample $X_0,X_{\Delta},\ldots$ (here $\Delta>0$ is fixed) from $X$ corresponding to the true parameter value $\theta_0$ is $\alpha$-mixing with strong mixing coefficients $\alpha_{\Delta}(k)$ satisfying the condition $\sum_{k=0}^{\infty}\alpha_{\Delta}(k)<\infty.$
\end{assump}

\begin{assump}
\label{cnd_smoothness}
The stationary density $\pi(\cdot;\cdot)$ satisfies the condition
\begin{equation*}
\forall t\in\mathbb{R}, \quad \left\{ \int_{\mathbb{R}} ( \pi^{(\alpha)}(x+t;\theta) - \pi^{(\alpha)}(x;\theta) )^2dx \right\}^{1/2} \leq L_{\theta} ,
\end{equation*}
for some constant $L_{\theta}>0$ (that may depend on $\theta$) and some integer $\alpha>3.$
\end{assump}

\begin{assump}
\label{cnd_kernel}
The kernel $K$ is symmetric and continuously differentiable, has support $[-1,1]$ and
satisfies the conditions
\begin{equation*}
\int_{-1}^1 K(u)du=1, \quad \int_{-1}^1 u^{l}K(u)du=0, \quad l=1,\ldots,\alpha.
\end{equation*}
Here $\alpha$ is the same as in Assumption \ref{cnd_smoothness}.
\end{assump}

\begin{assump}
\label{cnd_h}
The bandwidth $h=h_n$ depends on $n$ and $h\downarrow 0$ as $n\rightarrow\infty$ in such a way that
$nh^4\rightarrow\infty.$
\end{assump}

\begin{assump}
\label{cnd_w}
The weight function $w$ is nonnegative, continuously differentiable, bounded and integrable.
\end{assump}

\begin{assump}
\label{assump_characterisation}
The invariant density $\pi(\cdot;\cdot)$ solves the differential equation
\begin{equation}
\label{diffeq0}
\mu(x;\theta_0){\pi}(x) - \frac{1}{2}\left[ \sigma^2(x) {\pi}(x) \right]^{\prime}=0,
\end{equation}
where $\pi(\cdot)$ is the unknown function. Differentiability of $\sigma(\cdot)$ is also assumed.
\end{assump}

\begin{assump}
\label{cnd_drift_volatility}
The drift coefficient $\mu(\cdot;\cdot)$ is three times differentiable with respect to $\theta.$ The drift and dispersion coefficients and the corresponding derivatives are continuous functions of $x$ and $\theta.$ Furthermore, there exist functions $\widetilde{\mu}_j(\cdot),j=1,\ldots,4,$ such that
\begin{equation}
\label{mu_tilde}
\sup_{\theta\in\Theta} |\stackrel{(i)}{\mu}(x;\theta)| \leq \widetilde{\mu}_{i+1}(x), \quad \forall x\in\mathbb{R},
\end{equation}
for $i=0,1,2,3,$ and a function $\widetilde{\mu}_5(\cdot),$ such that
\begin{equation*}
\sup_{\theta\in\Theta} |\dot{\mu}^{\prime}(x;\theta)| \leq \widetilde{\mu}_{5}(x), \quad \forall x\in\mathbb{R}.
\end{equation*}
Here $\stackrel{(i)}{\mu}$ denotes the $i$th derivative of a function $\mu$ with respect to $\theta$ and $\stackrel{(0)}{\mu}(\cdot;\cdot)=\mu(\cdot;\cdot).$ Moreover, the functions
\begin{gather*}
\widetilde{\mu}_1^2(\cdot)w(\cdot), \quad \sigma^4(\cdot)w(\cdot), \quad \sigma^2(\cdot)(\sigma^{\prime}(\cdot))^2w(\cdot),\\
\widetilde{\mu}_2(\cdot)\widetilde{\mu}_1(\cdot)w(\cdot), \quad \widetilde{\mu}_5(\cdot)\sigma^2(\cdot)w(\cdot), \quad \widetilde{\mu}_2^2(\cdot)w(\cdot), \\
\sigma^2(\cdot)w(\cdot), \quad \sigma(\cdot)\sigma^{\prime}(\cdot)w(\cdot), \quad \widetilde{\mu}_2(\cdot)\sigma^2(\cdot)w^{\prime}(\cdot),\\
\widetilde{\mu}_3(\cdot)\widetilde{\mu}_1(\cdot)w(\cdot), \quad \widetilde{\mu}_3(\cdot)\sigma^4(\cdot)w(\cdot), \quad \widetilde{\mu}_3(\cdot)\widetilde{\mu}_1(\cdot)w(\cdot),\\
\widetilde{\mu}_3(\cdot)\sigma(\cdot)\sigma^{\prime}(\cdot)w(\cdot), \quad \widetilde{\mu}_3(\cdot)\sigma^2(\cdot)w(\cdot), \quad \widetilde{\mu}_3(\cdot)\widetilde{\mu}_2(\cdot)w(\cdot),\\
\widetilde{\mu}_4(\cdot)\widetilde{\mu}_1(\cdot)w(\cdot), \quad
\widetilde{\mu}_4(\cdot)\sigma(\cdot)\sigma^{\prime}(\cdot)w(\cdot), \quad \widetilde{\mu}_4(\cdot)\sigma^2(\cdot)w(\cdot),\\
\widetilde{\mu}_2(\cdot)\sigma^2(\cdot)w(\cdot), \quad \widetilde{\mu}_2(\cdot)\sigma(\cdot)\sigma^{\prime}(\cdot)w(\cdot), \quad \widetilde{\mu}_4(\cdot)\sigma^2(\cdot)w(\cdot)
\end{gather*}
are bounded and integrable. Finally,
$\lim_{|x|\rightarrow\infty}\widetilde{\mu}_2(x)w(x)\sigma^2(x)=0.$
\end{assump}

\section{Remarks on assumptions}
\label{remarks}

In this section we provide remarks on the assumptions made in Section \ref{assumptions}.

\begin{rem}
\label{univ} In Assumption \ref{cnd_theta} we assume that the parameter $\theta$ is univariate. This assumption is made for simplicity of the proofs only and the results of the paper can also be generalised to the case when $\theta$ is multivariate. Compactness of the parameter space $\Theta$ guarantees existence of our estimator $\hat{\theta}_n.$
\end{rem}

\begin{rem}
\label{rem_sol} In this remark we deal with Assumption \ref{assump_sol}. A standard condition that guarantees existence and uniqueness of a strong solution to \eqref{sde} is a Lipschitz and linear growth condition on the coefficients $\mu(\cdot;\theta)$ and $\sigma(\cdot)$ together with an assumption that $\ex[\xi^2]<\infty,$ see e.g.\ Theorem 1 on p.\ 40 in \cite{gikhman} or Theorem 2.9 on p.\ 289 in \cite{karatzas}. The same condition also implies that $X$ will be a Markov process, see e.g.\ Theorem 1 on p.\ 66 in \cite{gikhman}, time-homogeneity of which can be shown as on pp.\ 106--107 in \cite{gikhman}. Moreover, $X$ will be a diffusion process, see Theorem 2 on p.\ 67 in \cite{gikhman}. Conditions for ergodicity of $X$ and existence of the invariant density are given e.g.\ in Theorem 3 on p.\ 143 in \cite{gikhman}, while those for existence of the transition density $p(\Delta,x,y;\theta),$ as well as its characterisations can be found in \S 13 of Chapter 3 of \cite{gikhman}. Ergodicity is a standard assumption in parameter estimation problems for diffusion processes from discrete time observations, at least in the problems with $\Delta$ fixed. The condition in Assumption \ref{assump_sol} that the support of $\pi(\cdot;\theta)$ for every $\theta\in\Theta$ equals $\mathbb{R}$ is a purely technical one and is needed only in order to avoid extra technicalities when dealing with boundary bias effects characteristic of kernel density estimators. This condition is for instance satisfied in case when the process $X$ is an Ornstein-Uhlenbeck process,
\begin{equation}
\label{ou}
dX_t=-\theta X_t dt+\sigma dW_t,
\end{equation}
with $\theta>0$ and known $\sigma,$ because in this case
$\pi(x;\theta)$ is a normal density with mean $0$ and variance $\sigma^2/(2\theta),$ see Proposition 5.1 on p.\ 219  in \cite{karlin} or Example 4 on p.\ 221 there. For more information on the Ornstein-Uhlenbeck process see Example 6.8 on p.\ 358 in \cite{karatzas} or results on the Ornstein-Uhlembeck process scattered throughout \cite{karlin}. In the financial literature a slight generalisation of the Ornstein-Uhlenbeck process is used to model the dynamics of the short interest rate and the corresponding model is known under the name of the Vasicek model, see for instance pp.\ 350--355 in \cite{musiela}.
A general case when the support of $\pi(\cdot;\theta)$ does not coincide with $\mathbb{R}$ as for instance for the CIR process, where it is equal to $(0,\infty),$ can be dealt with using the same approach as in the present work in combination with a boundary bias correction method that uses a kernel with special properties, see e.g.\ \cite{gasser}. An alternative in the case when the state space of $X$ is $(0,\infty)$ is to use the transformation $Y_t=\log X_t.$ The process $Y$ will have the state space $\mathbb{R}$ and its governing stochastic differential equation can be obtained through It\^{o}'s formula. \qed
\end{rem}

\begin{rem}
\label{rem_mixing}
Assumption \ref{cnd_mixing} implies certain restrictions on the rate of decay of $\alpha$-mixing coefficients $\alpha_{\Delta}(k).$ Conditions yielding information on their rate of decay can be obtained for instance from the corresponding results for $\beta$-mixing coefficients $\beta(s)$ for the process $X.$ A $\beta$-mixing coefficient $\beta(s)$ (attributed to Kolmogorov in \cite{rozanov} and alternatively called the absolute regularity coefficient) for the process $X$ is defined as follows,
\begin{equation*}
\beta(s)=\sup_{t\geq 0}\ee\left[ \esssup_{B\in \mathfrak{F}_{\geq t+s}} |\pp(B|\mathfrak{F}_{\leq t})-\pp(B)| \right],
\end{equation*}
where $\mathfrak{F}_{\geq s+t}=\sigma(X_u,u\geq s+t),$ $\mathfrak{F}_{\leq t}=\sigma(X_u,u\leq t)$ and $\pp(\cdot|\mathfrak{F}_{\leq t})$ is the regular conditional probability on $\mathfrak{F}_{\geq t+s}$ given $\mathfrak{F}_{\leq t}$ (the latter will exist in our context by Theorem 3.19 on pp.\ 307--308 in \cite{karatzas}). Theorem 1 in \cite{veretennikov} gives a sufficient condition  on the drift coefficient (satisfied for instance in the case of the Ornstein-Uhlenbeck process), which entails a bound
\begin{equation}
\label{polbound}
\beta(s) \leq C \frac{1}{(1+s)^{\kappa+1}},
\end{equation}
where $C$ is a constant independent of $s$ and $\kappa$ depends in a simple way on the drift coefficient. An $\alpha$-mixing coefficient $\alpha(s)$ (introduced in \cite{rosenblatt}) is defined as
\begin{equation*}
\alpha(s)=\sup_{t\geq 0} \sup_{A\in \mathfrak{F}_{\leq t},B\in \mathfrak{F}_{\geq t+s} } |\mathbb{P}(AB)-\mathbb{P}(A)\mathbb{P}(B)|.
\end{equation*}
The following inequality is well-known: $2\alpha(s)\leq\beta(s),$ see Proposition 1 on p.\ 4 in \cite{doukhan}. Since one trivially has $\alpha_{\Delta}(k)\leq\alpha(k\Delta),$ it follows that $\alpha_{\Delta}(k)\leq(1/2)\beta(k\Delta).$ Therefore, by \eqref{polbound} in this case $\sum_{k=0}^{\infty} \alpha_{\Delta}(k)<\infty,$ i.e.\ the requirement in Assumption \ref{cnd_mixing} will hold. \qed
\end{rem}

\begin{rem}
\label{rem_smoothness} This remark deals with Assumption \ref{cnd_smoothness}. Viewing $\theta$ as fixed, conditions under which the invariant density $\pi(x;\theta)$ is infinitely differentiable with respect to $x$ can be found in Theorem 3 of \cite{kusuoka}. In simple cases like that of the Ornstein-Uhlenbeck \eqref{ou}, the regularity assumptions can and have to be checked by a direct calculation. The requirement that $\alpha>3$ is needed in order to establish Theorem \ref{rootn}. Under Assumption \ref{cnd_smoothness} the stationary density $\pi(\cdot;\theta)$ belongs to the Nikol'ski class of functions $\mathcal{H}(\alpha,L)$ as defined e.g.\ in Definition 1.4 in \cite{tsybakov}.  Another possibility is to assume that the invariant density $\pi(\cdot;\theta)$ is $\alpha$ times differentiable with continuous, bounded and square integrable derivative of order $\alpha,$ see e.g.\ paragraph VI.4 on p.\ 79 and Theorem VI.5 on p.\ 80 in \cite{bosq}. In case the weight function $w$ has a compact support, Lemma \ref{mise} (which is a basic result used in the proofs of the main statements of the paper) can also be proved under the assumption
\begin{equation*}
\forall x,t\in\mathbb{R}, \quad |\pi^{(\alpha)}(x+t;\theta) - \pi^{(\alpha)}(x;\theta) | \leq L_{\theta} ,
\end{equation*}
i.e.\ the assumption that the density $\pi(\cdot;\theta)$ belongs to the H\"{o}lder class $\Sigma(\alpha,L)$ as defined e.g.\ in Definition 1.2 in \cite{tsybakov}. However, if $w$ has compact support, in our analysis we will not be using all the information supplied by the stationary density. This might require stronger conditions on the drift and dispersion coefficients $\mu(\cdot;\cdot)$ and $\sigma(\cdot)$ in order for the identifiability condition \eqref{identifiability} in the statement of Theorem \ref{rootn} hold true and hence it is preferable to keep $w$ general. \qed
\end{rem}

\begin{rem}
\label{rem_kernel}
Assumption \ref{cnd_kernel} is a standard condition in kernel estimation, see e.g.\ p.\ 13 in \cite{tsybakov}. The kernel $K$ satisfying Assumption \ref{cnd_kernel} is called a kernel of order $\alpha.$ For a method of its construction see Section 1.2.2 in \cite{tsybakov}. \qed
\end{rem}

\begin{rem}
\label{rem_h}
Assumption \ref{cnd_h} is needed in order to establish consistency of the estimators $\hat{\pi}(\cdot)$ and $\hat{\pi}^{\prime}(\cdot),$ see Lemma \ref{mise}. \qed
\end{rem}

\begin{rem}
\label{rem_w}
This remark deals with Assumption \ref{cnd_w}. In practice when implementing the estimator of $\hat{\theta}_n,$ one would typically use $w$ with compact support. See Section \ref{simulations} for details. \qed
\end{rem}

\begin{rem}
\label{rem_banon}
Sufficient conditions guaranteeing \eqref{diffeq0} in Assumption \ref{assump_characterisation} can be gleaned from \cite{banon}, see Lemma 3.2 there, and involve regularity conditions on the drift coefficient $\mu(\cdot;\cdot)$ and dispersion coefficient $\sigma(\cdot).$ Note that for simple cases like the Ornstein-Uhlenbeck process \eqref{ou}, where an explicit formula for the invariant density is available, Assumption \ref{assump_characterisation} can also be verified directly. \qed
\end{rem}

\begin{rem}
\label{rem_drift_dispersion}
Conditions on the drift and dispersion coefficients made in Assumption \ref{cnd_drift_volatility} are used to prove asymptotic results of the paper. With an appropriate choice of the weight function $w(\cdot)$ they are satisfied in a number of interesting examples, for instance in the case of the Ornstein-Uhlenbeck process \eqref{ou} with $\theta>0$ unknown and $\sigma$ known. Examination of the proofs shows that complicated conditions in Assumption \ref{cnd_drift_volatility} can be significantly simplified if the weight function $w$ is taken to have a compact support. Note also that because of a great flexibility in selection of the weight function $w,$ Assumption \ref{cnd_drift_volatility} will be satisfied in a large number of examples. \qed
\end{rem}

\section{Consistency}
\label{results}

Let $K$ be a kernel function and a number $h>0$ (that depends on $n$) be a bandwidth. To construct our estimator of $\theta_0,$ we first need to construct a nonparametric estimator of the stationary density $\pi(\cdot;\theta_0).$ The stationary density $\pi(\cdot;\theta_0)$ will be estimated by a kernel density estimator
\begin{equation*}
\hat{\pi}(x)=\frac{1}{(n+1)h} \sum_{j=0}^n K\left( \frac{x-Z_j}{h} \right),
\end{equation*}
while $\hat{\pi}^{\prime}(\cdot)$ will serve as an estimator of $\pi^{\prime}(\cdot;\theta_0)$ (we assume that $K(\cdot)$ is differentiable). Kernel density estimators are among the most popular nonparametric density estimators, see e.g.\ Chapter 1 in \cite{tsybakov} for an introduction in the i.i.d.\ case and Section 2 in Chapter 4 of \cite{gyorfi} for the case of dependent identically distributed observations.

In the sequel we will need to know the convergence rate of the estimator $\hat{\pi}(\cdot)$ and its derivative $\hat{\pi}^{\prime}(\cdot)$ in the weighted $L_2$-norm with weight function $w.$ As usual in nonparametric density estimation, to that end some degree of smoothness of the stationary density $\pi(\cdot;\cdot),$ as well as appropriate conditions on the kernel $K,$ bandwidth $h$ and weight function $w$ are needed. These are supplied in Section \ref{assumptions}. Furthermore, to establish useful asymptotic properties of the estimators $\hat{\pi}(\cdot)$ and its derivative $\hat{\pi}^{\prime}(\cdot),$ some further assumptions on the observations have to be made. We will assume that the sequence $Z_j=X_{j\Delta}$ is strongly mixing with mixing coefficients satisfying a condition spelled out in Section \ref{assumptions}.

The following result holds true.

\begin{lemma}
\label{mise}
Under Assumptions \ref{cnd_theta}--\ref{cnd_w} we have
\begin{equation}
\label{mise1}
\ex\left[ \int_{\mathbb{R}} (\hat{\pi}(x)-\pi(x;\theta_0))^2 w(x) dx \right] \lesssim h^{2\alpha}+\frac{1}{nh^{2}},
\end{equation}
and
\begin{equation}
\label{mise3}
\ex\left[ \int_{\mathbb{R}} (\hat{\pi}^{\prime}(x)-\pi^{\prime}(x;\theta))^2 w(x) dx \right] \lesssim h^{2(\alpha-1)}+\frac{1}{nh^{4}}.
\end{equation}
\end{lemma}

\begin{rem}
\label{rem_mise}
The bound in inequality \eqref{mise1}, and by extension in inequality \eqref{mise3}, can be sharpened by using more refined arguments in the proof of Lemma \ref{mise}, such as Theorem 3 on p.\ 9 in \cite{doukhan}. However, the `usual' order bound on the mean integrated squared error in kernel density estimation for i.i.d.\ observations when the unknown density is `smooth of order $\alpha$', i.e.\
\begin{equation}
\label{mise_usual}
\ex\left[ \int_{\mathbb{R}} (\hat{\pi}(x)-\pi(\theta,x))^2 w(x) dx \right] \lesssim h^{2\alpha}+\frac{1}{nh},
\end{equation}
see e.g.\ Theorem 1.3 in \cite{tsybakov}, does not seem to be obtainable without further conditions. For dependent observations the bound \eqref{mise_usual} is true by Theorem 3.3 in \cite{viennet}, which, however, is proved under $\beta$-mixing assumption on observations (which is stronger than $\alpha$-mixing) and some extra condition on the $\beta$-mixing coefficients (see also \cite{gourieroux} and \cite{krist} for related results). The proof of a similar result in \cite{vieu} under $\alpha$-mixing assumption and some complicated conditions on the mixing coefficients, see Theorem 2.2 there, is unfortunately incorrect: the assumption (2.3b) in that paper is impossible to satisfy unless the observations are independent, formula (A.9) contains a mistake and formula (9.2) requires some further conditions in order to hold. \qed
\end{rem}

Let the estimator $\hat{\theta}_n$ of $\theta_0$ be defined by \eqref{est_theta}.

\begin{rem}
\label{rem_minimiser}
Under our assumptions in Section \ref{assumptions} the criterion function $R_n(\theta)$ from \eqref{crf} is a continuous function of $\theta$ and hence by compactness of $\Theta$ a minimiser of $R_n(\theta)$ over $\theta\in\Theta$ exists. Consequently, so does the estimator $\hat{\theta}_n,$ although it might be non-unique. Moreover, the estimator $\hat{\theta}_n$ will be a measurable function of the observations $Z_0,Z_1,\ldots,Z_n$ and hence when dealing with convergence properties of $\hat{\theta}_n,$ the use of outer probability, will not be needed. Observe that $\hat{\theta}_n,$ being defined through a minimisation procedure, is an M-estimator, see e.g.\ Chapter 5 in \cite{vdvaart}. \qed
\end{rem}

\begin{rem}
\label{rem_esteq}
An approach to parameter estimation for stochastic differential equations that is based on estimating equations as described in Section \ref{intro} in practice might suffer from non-uniqueness of a parameter estimate, i.e.\ non-uniqueness of a root of the estimating equations. `Wrong' selection of a root of the estimating equations might even render the estimator inconsistent, see e.g.\ remarks on pp.\ 70--71 in \cite{vdvaart}. For a thorough discussion of the multiple root problems and possible remedies for them see \cite{small}. On the other hand, an approach based on maximisation of a criterion function, such as the one advocated in the present work, is less prone to failures of this type. \qed
\end{rem}

\begin{rem}
\label{rem_estex}
In many interesting models, in particular in those where the drift coefficient $\mu(\cdot;\cdot)$ is linear in $\theta,$ the estimator $\hat{\theta}_n$ will have a simple expression. For instance, one can check that for the Ornstein-Uhlenbeck process \eqref{ou} with $\theta>0$ unknown and and $\sigma=1$ known, the estimator $\hat{\theta}_n$ of the true parameter value $\theta_{0}$ is given by
\begin{equation}
\label{thetaou}
\hat{\theta}_n = - \frac{1}{2} \frac{ \int_{\mathbb{R}} x\hat{\pi}(x)  \hat{\pi}^{\prime}(x) w(x) dx }{ \int_{\mathbb{R}} x^2 \hat{\pi}^2(x) w(x)dx  }.
\end{equation}
Compare this expression to the rather complex and nonlinear score function for the same model as given on p.\ 77 in \cite{kessler}, which is used as an estimating function when $\theta$ is estimated by the maximum likelihood method and which requires use of some numerical root finding technique for the computation of the estimator. A general conclusion that can be drawn from this and other similar examples is that our approach in many interesting examples will provide explicit estimators. However, it should be noted that from the point of view of numerical stability, evaluation of the estimator through expressions such as in \eqref{thetaou} cannot be recommended in practice. Rather, one should approximate the criterion function $R_n(\cdot)$ through a Riemann sum and next compute from this approximation the estimator $\hat{\theta}_n$ as a weighted least squares estimator. When $\mu(\cdot;\cdot)$ is linear in $\theta,$ the problem will further reduce to a standard task of computing the weighted least squares estimator in the linear regression model. Finally, we remark that with a proper implementation of the nonparametric kernel estimators $\hat{\pi}_n(\cdot)$ and $\hat{\pi}_n^{\prime}(\cdot),$ computational effort for their evaluation is very modest; see e.g.\ \cite{marron}. \qed
\end{rem}

\begin{rem}
\label{rem_kessler}
A desire to have simple expressions for estimators based on estimating equations in \cite{kessler} at times leads to unnatural assumptions on the parameter space $\Theta.$ For instance, in Section 6.4 in \cite{kessler} in the model
\begin{equation*}
\begin{cases}
dX_t=-\theta X_t dt+\sqrt{\theta+X_t^2}dW_t,\\
X_0=\xi,
\end{cases}
\end{equation*}
in order to accommodate a simple looking estimator of the true parameter $\theta_0,$ $\theta_0>7/2$ has to be assumed, while the more general condition $\theta_0>0$ appears to be more natural here. On the other hand, the assumption $\theta_0>7/2$ is not needed for our estimator $\hat{\theta}_n$ and $\theta_0>0$ suffices (this model formally does not fit into our framework, because the unknown parameter $\theta$ is also included in the dispersion coefficient of the stochastic differential equation. However, our asymptotic analysis holds for this model as well). \qed
\end{rem}

It can be expected that as $n\rightarrow\infty,$ for every $\theta\in\Theta$ the criterion function $R_n(\theta)$ converges in some appropriate sense  to the limit criterion function
\begin{equation*}
R(\theta)=\int_{\mathbb{R}}\left( \mu(x;\theta){\pi}(x;\theta_0) - \frac{1}{2}\left[ \sigma^2(x) {\pi}(x;\theta_0) \right]^{\prime} \right)^2 w(x)dx.
\end{equation*}
Note that by our assumptions $R(\theta_0)=0$ and that $R(\theta)\geq 0$ for $\theta\in\Theta.$ Hence the parameter value $\theta_0$ is a minimiser of the asymptotic criterion function $R(\theta)$ over $\theta\in\Theta.$ Under suitable identifiability conditions it can be ensured that $\theta_0$ is the unique minimiser of $R(\theta).$ Next, if convergence of $R_n(\theta)$ to $R(\theta)$ is strong enough, a minimiser of $R_n(\theta)$ will converge to a minimiser of $R(\theta)$. Said another way, $\hat{\theta}_n$ will be consistent for $\theta_0.$ This is a standard approach to prove consistency of M-estimators, see e.g.\ Section 5.2 in \cite{vdvaart}.

In order to carry out the above programme for the proof of consistency of $\hat{\theta}_n,$ we need that the drift coefficient $\mu(\cdot;\cdot)$ and the dispersion coefficient $\sigma(\cdot)$ satisfy certain regularity conditions. These are listed in Section \ref{assumptions}. Then the following theorem holds true.

\begin{thm}
\label{consistency}
Under Assumptions \ref{cnd_theta}--\ref{cnd_drift_volatility} and the additional identifiability condition
\begin{equation}
\label{identifiability}
\forall \varepsilon>0, \quad \inf_{\theta:|\theta-\theta_0|\geq\varepsilon} R(\theta) > R(\theta_0),
\end{equation}
the estimator $\hat{\theta}_n$ is weakly consistent: $\hat{\theta}_n \convp \theta_0.$
\end{thm}

\begin{rem}
\label{rem_identifiability}
The identifiability condition \eqref{identifiability} is standard in M-estimation, see e.g.\ a discussion in Section 5.2 in \cite{vdvaart}. It means that a point of minimum of the asymptotic criterion function is a well-separated point of minimum. Since under our conditions the asymptotic criterion function $R(\theta)$ is a continuous function of $\theta$ and $\Theta$ is compact, uniqueness of a global minimiser of $R(\theta)$ over $\theta$ will imply \eqref{identifiability}, cf.\ Problem 5.27 on p.\ 84 in \cite{vdvaart}. As one  particular example, one can check that condition \eqref{identifiability} is satisfied for the Ornstein-Uhlenbeck process \eqref{ou}, assuming that $\theta$ is unknown, while $\sigma$ is known. \qed
\end{rem}

\begin{thm}
\label{rootn}
Let the assumptions of Theorem \ref{consistency} hold and let additionally $\theta_0$ be an interior point of $\Theta.$ If $h\asymp n^{-\gamma}$ with $\gamma=1/(2\alpha)$ and
$\ddot{R}(\theta_0)\neq 0,$
then 
$\sqrt{n}(\hat{\theta}_n-\theta_0)=O_{\mathbb{P}}(1).$
\end{thm}

\begin{rem}
\label{rem_secondder}
The assumption $\ddot{R}(\theta_0)\neq 0$ is satisfied in a number of important examples, for instance in the case of the Ornstein-Uhlenbeck process \eqref{ou} with $\theta>0$ unknown and known $\sigma.$ \qed
\end{rem}

\begin{rem}
\label{rem_switch} Under appropriate conditions, by the same method as studied in the present work, one can also handle the case when the drift coefficient $\mu(\cdot;\theta)$ does not depend on parameter $\theta,$ while the dispersion coefficient does. \qed
\end{rem}

\begin{rem}
\label{rem_extension}
In the present paper we assumed that the dispersion coefficient $\sigma(\cdot)$ was known. In practice this is not always a realistic assumption. A possible extension of the smooth and match method to this more general setting is to assume that $\sigma(\cdot)$ is a totally unknown function, to estimate it nonparametrically and next to define an estimator of the parameter of interest $\theta_0$ again via an expression \eqref{est_theta}, but with $\sigma(\cdot)$ replaced by its nonparametric estimator $\hat{\sigma}(\cdot)$ in $R_n(\theta).$ Under appropriate assumptions this approach should again yield a $\sqrt{n}$-consistent estimator of $\theta_0,$ although some nontrivial technicalities can be anticipated. \qed
\end{rem}

\begin{rem}
\label{rem_dband} Theorem \ref{rootn} holds also for bandwidth sequences $h\asymp n^{-\gamma}$ with $\gamma$ other than $1/(2\alpha).$ However, $\gamma$ cannot be arbitrary, for this might lead to violation of consistency of $\hat{\pi}(\cdot)$ and $\hat{\pi}^{\prime}(\cdot),$ see Lemma \ref{mise}. The condition on the bandwidth sequence in the statement of Theorem \ref{rootn} is of an asymptotic nature and is not directly applicable in practice. In practical applications a simple method called the quasi-optimality method is likelily to produce reasonable results, see e.g.\ \cite{reiss2} for more information. See also the results of the simulation examples considered in Section \ref{simulations}. \qed
\end{rem}

\section{Asymptotic normality}
\label{normality}
Examination of the proof of Lemma \ref{t1t2t3t4} in Appendix \ref{appendix}, on which the proof of Lemma \ref{lemma_rootn1} and eventually that of Theorem \ref{rootn} relies, shows that under appropriate extra conditions not only $\sqrt{n}$-consistency of the estimator $\hat{\theta}_n,$ but also its asymptotic normality can be established.

Let
\begin{equation*}
v(x)=2\dot{\mu}(x;\theta_0)\mu(x;\theta_0)\pi(x;\theta_0)w(x)
+\left[ \dot{\mu}(x;\theta_0) \pi(x;\theta_0) w(x) \right]^{\prime} \sigma^2(x).
\end{equation*}
The following result holds true.

\begin{thm}
\label{thm_an}
Let the assumptions of Theorem \ref{consistency} hold (with Assumption \ref{rem_smoothness} strengthened to the requirement $\alpha>4$) and let additionally $\theta_0$ be an interior point of $\Theta.$ Assume that $h\asymp n^{-\gamma}$ with
\begin{equation*}
\frac{1}{2\alpha}<\gamma<\frac{1}{8}.
\end{equation*}
If
\begin{equation}
\label{secondder2}
\ddot{R}(\theta_0)\neq 0, \quad \var[v(Z_0)]+2\sum_{j=1}^{\infty}\cov[v(Z_0),v(Z_j)]>0, \quad \|v^{(\alpha)}\|_{\infty}<\infty,
\end{equation}
and for some $\delta>0,$
\begin{equation}
\label{an_mixing}
\ex[|Z_j|^{2+\delta}]<\infty, \quad \sum_{k=1}^{\infty} (\alpha_{\Delta}(k))^{\delta/(2+\delta)}<\infty,
\end{equation}
then
\begin{equation*}
\sqrt{n+1}(\hat{\theta}_n-\theta_0)\convd \mathcal{N}(0,s^2).
\end{equation*}
Here
\begin{equation*}
s^2=\frac{ \var[v(Z_0)]+2\sum_{j=1}^{\infty}\cov[v(Z_0),v(Z_j)] }{(\ddot{R}(\theta_0))^2}.
\end{equation*}
\end{thm}

\section{One-step Newton-Raphson type procedure}
\label{onestep}

Although according to Theorems \ref{rootn} and \ref{thm_an} the estimator $\hat{\theta}_n$ is $\sqrt{n}$-consistent and even asymptotically normal, it is obviously not necessarily asymptotically the best one, which in the present model and observation scheme is typically the case for Z-estimators based on martingale estimating equations as well. Here we interpret asymptotically the best estimator as the one that is regular and has the smallest possible asymptotic variance among all regular estimators, see e.g.\ Chapter 8 in \cite{vdvaart} for an exposition of the asymptotic efficiency theory in the i.i.d.\ setting. Under regularity conditions the maximum likelihood estimator achieves the efficiency bound. As far as Z-estimators in diffusion models are concerned, a line of research in the literature is to try to choose estimating equations within a certain class of functions in an optimal way, see e.g.\ \cite{m_sorensen}, \cite{jacobsen} and \cite{kessler}. However, most of the work in this direction deals with the high frequency data setting where $\Delta=\Delta_n\rightarrow 0$ as $n\rightarrow\infty.$ In our case optimal choice of the estimating equations would correspond to the problem of an optimal choice of the weight function $w(\cdot)$ within a certain class of weight functions. This is not an easy problem to solve and it is a priori not clear whether this approach would lead to a simple and feasible optimal weight function $w_{opt}.$ A possibly better and more direct approach to improving asymptotic performance of the estimator $\hat{\theta}_n$ is to use it as a starting point of a one-step Newton-Rapshon type procedure. The idea is well-known in statistics, see e.g.\ Section 5.7 in \cite{vdvaart}, and is as follows: consider an estimating equation $\Psi_n(\theta)=0.$ Given a preliminary estimator $\hat{\theta}_n,$ define a one-step estimator $\overline{\theta}_n$ of $\theta_0$ as a solution in $\theta$ to the equation
\begin{equation}
\label{eq1step}
\Psi_n(\hat{\theta}_n)+ \dot{\Psi}_n(\hat{\theta}_n) (\theta-\hat{\theta}_n)=0,
\end{equation}
where $\dot{\Psi}_n(\cdot)$ is the derivative of $\Psi_n(\cdot)$ with respect to $\theta.$ This corresponds to replacing $\Psi_n(\theta)$ with its tangent at $\hat{\theta}_n$ and when iterated several times, each time using as a new starting point the previously found solution to \eqref{eq1step}, is known in numerical analysis under the name of the Newton (or Newton-Raphson) method, see e.g.\ Section 2.3 in \cite{burden}. This method is used to find zeroes of nonlinear equations. In statistics, on the other hand, just one such iteration suffices to obtain an estimator that is as good asymptotically as the one defined by the estimating equation $\Psi_n(\theta)=0,$ provided the preliminary estimator is already $\sqrt{n}$-consistent (a precise result can be found in Theorem 5.45 in \cite{vdvaart}). A computational advantage of a one-step approach over a more direct maximum likelihood approach is that often a preliminary $\sqrt{n}$-consistent estimator is easy to compute, while the computational time required for one Newton-Raphson type iteration step is negligible. 

Under suitable conditions one can use in the capacity of $\Psi_n$ the martingale estimating functions, see e.g.\ \cite{m_sorensen}, or even the score function $S_n(\theta)$ (i.e.\ a gradient of the likelihood function with respect to the unknown parameter $\theta$), provided the required derivatives of $\Psi_n$ can be evaluated either analytically or numerically in a quick and numerically stable way. The estimator $\hat{\theta}_n$ can thus be upgraded to an asymptotically efficient one. We omit a detailed discussion and a precise statement to save space and will simply note that the regularity conditions required to justify the one-step method are mild enough in our case (as an example, they are satisfied in the case of the Ornstein-Uhlenbeck process).

\section{Simulations}
\label{simulations}

In this section we present results of a small simulation study that we performed using the Ornstein-Uhlenbeck process \eqref{ou} as a test model. This study is in no way exhaustive and the results obtained merely serve as an illustration of the theoretical results from Sections \ref{results}--\ref{onestep}.

Three required components for the construction of our estimator $\hat{\theta}_n$ from \eqref{est_theta} are the weight function $w(\cdot),$ the kernel $K$ and the bandwidth $h.$ As a weight function we used a suitably rescaled version of the function
\begin{equation*}
\lambda_{c,\beta}(x)=
\begin{cases}
1, & \text{if $|x|\leq c,$}\\
\exp[-\beta \exp[-\beta/(|x|-c)^2]/(|x|-1)^2], & \text{if $c<|t|<1,$}\\
0, & \text{if $|x| \geq 1,$}
\end{cases}
\end{equation*}
with constants $c$ and $\beta$ equal to $0.7$ and $0.5,$ respectively. This weight function was already used in simulation examples in \cite{gugu}. The rationale for its use is simple: $w$ will be equal to one on a greater part of its support, which comes in handy in computations, while at the same time being smooth. As a kernel we used
\begin{equation*}
K(x)=\left( \frac{105}{64} - \frac{315}{64} x^2 \right) (1-x^2)^2 1_{[|x|\leq 1]},
\end{equation*}
which was also employed in simulation examples in \cite{gugu} and yielded good results there. Finally, in all our examples the bandwidth was selected through the so-called quasi-optimality approach by computing the estimates $\hat{\theta}_n=\hat{\theta}_{n,h}$ for a range of different bandwidths $h$ and then picking the one that brought the least change to the next estimate. In greater detail, for a sequence of bandwidths $\{h^{(i)}\}$ we chose the bandwidth $\hat{h}$ such that
\begin{equation*}
\hat{h}=\operatorname{argmin}_{h^{(i)}} \|\hat{\theta}_{n,h^{(i+1)}}-\hat{\theta}_{n,h^{(i)}}\|
\end{equation*}
and next computed the estimate $\hat{\theta}_{n,\hat{h}}.$ In order not to clutter the notation, in the sequel we will omit the dependence of $\hat{\theta}_{n,\hat{h}}$ on $\hat{h}$ and will simply write $\hat{\theta}_n.$ \cite{reiss2} contains theoretical justification for this method of smoothing parameter selection in nonparametric estimation problems. 

Our goal was to compare the behaviour of our estimator $\hat{\theta}_n,$  the one-step estimator $\overline{\theta}_n$ which was using $\hat{\theta}_n$ as a preliminary estimator, the estimator based on a simple estimating function from formula (29) in \cite{kessler} given by the expression
\begin{equation*}
\theta^{*}_n=\frac{n}{2 \sum_{j=0}^{n-1} X_{j\Delta}^2},
\end{equation*}
and the maximum likelihood estimator $\widetilde{\theta}_n.$ Since the practical performance of the maximum likelihood estimator $\widetilde{\theta}_n$ in the case of the Ornstein-Uhlenbeck process is quite good, while the loss in asymptotic efficiency for the estimator $\theta^{*}_n$ in comparison to $\widetilde{\theta}_n$ is small, the competition with these two estimators was a tough task for our estimator $\hat{\theta}_n.$

All the computations were performed in Wolfram Mathematica 8.0, see \cite{mathematica}. Simulating samples from the Ornstein-Uhlenbeck process is straightforward, since it is an AR(1) process. We took $\theta_0=2$ and $\sigma=1$ and simulated from the process $X$ samples of sizes $100$ and $200$ (thus $n=99$ and $199$) with intervals between successive observations $\Delta=0.01,0.05,0.1$ and $1.$

As a criterion for comparison of different estimators the mean squared error was used. For fixed $\Delta$ and $n$ and for $k=200$ different samples we computed the estimates $\hat{\theta}_n,\overline{\theta}_n,\theta^{*}_n$ and $\widetilde{\theta}_n$ and then for each of $k=200$ estimates $\hat{\theta}_n,\overline{\theta}_n,\theta^{*}_n,\widetilde{\theta}_n$ we evaluated the corresponding mean squared error, that is the sum of the sample variance and sample bias squared (sample mean minus the true parameter value $\theta_0=2$ squared). The support of the weight function $w(\cdot)$ was taken to be the interval $[-1.4,1.4],$ which roughly corresponds to the interval $[-3s_n,3s_n],$ where $s_n$ is the sample standard deviation of the observations. The results obtained from our simulations are reported in Table \ref{table1}, where we also included the theoretical optimal value $\text{EB}$ for the mean squared error that can be obtained from the asymptotic efficiency bound, see Example 3.2 and formula (12) in \cite{kessler}.
\begin{table}
\caption{Mean squared errors for the estimates $\hat{\theta}_n,\overline{\theta}_n,\theta^{*}_n,\widetilde{\theta}_n$ together with the optimal value $\text{EB}$ obtained from the asymptotic efficiency bound in the case of the Ornstein-Uhlenbeck process \eqref{ou} with $\theta_0=2$ and $\sigma=1.$}
\label{table1}
\begin{center}
\begin{tabular}{rr|rrrrr}
$\Delta$ & $n$ & $\hat{\theta}_n$ & $\overline{\theta}_n$ & $\theta^{*}_n$ & $\widetilde{\theta}_n$ & $\text{EB}$\\
\hline
$0.01$ & $99$ & $1.900$ & $11.24$ & $8.545$ & $11.28$ & $4.001$\\
       &$199$ & $2.152$ & $3.774$ & $3.474$ & $3.776$ & $2.000$\\
       \hline
$0.05$ & $99$ & $1.061$ & $1.384$ & $1.371$ & $1.394$ & $0.803$\\
       &$199$ & $0.578$ & $0.647$ & $0.615$ & $0.651$ & $0.401$\\
       \hline
$0.1$  & $99$ & $0.663$ & $0.697$ & $0.677$ & $0.701$ & $0.405$\\
       &$199$ & $0.291$ & $0.204$ & $0.206$ & $0.205$ & $0.203$\\
       \hline
$1$    & $99$ & $0.155$ & $0.067$ & $0.079$ & $0.070$ & $0.080$\\
       &$199$ & $0.093$ & $0.040$ & $0.042$ & $0.040$ & $0.040$
\end{tabular}
\end{center}
\end{table}
A conclusion (modulo the Monte Carlo simulation errors) that lends itself from this table is that for small $\Delta$ the estimator $\hat{\theta}_n$ seems to either outperform other estimators, or to perform just as well as other estimators, but once $\Delta$ and $n$ are sufficiently large, it is itself outperformed by other estimators (our conclusions are also supported by some other simulations not reported here). Curiously enough, for $\Delta=0.01$ and $n=99$ the estimator $\hat{\theta}_n$ beats the asymptotic efficiency bound, although of course its performance is not (and cannot be) particularly good in this case. It is also interesting to note that the maximum likelihood estimator is not the best estimator in all the cases, which should not be surprising, for its superiority over other estimators is in the asymptotic sense only (it is also known to be strongly biased for small $n\Delta,$ see e.g.\ \cite{tang}). Note that whenever the maximum likelihood estimator $\widetilde{\theta}_n$ performs well, so does the one-step estimator $\overline{\theta}_n,$ which in general seems to yield virtually indistinguishable results. Another general remark is that for $n$ fixed all the estimators tend to perform better for larger values of $\Delta.$ An intuitive explanation of this fact is that increasing $\Delta$ decreases the degree of dependence between different observations, which coupled with the fact that in the case of the Ornstein-Uhlenbeck process the marginal distributions of the process $X$ contain enough information on the parameter $\theta_0,$ improves the estimation quality.

In conclusion, keeping in mind that in our simulation study we used a very simple bandwidth selector and a weight function $w(\cdot),$ the choice of which was primarily motivated by simplicity considerations, the performance of our estimator $\hat{\theta}_n$ can be deemed as satisfactory.

\section{Proofs}
\label{proofs}

\begin{proof}[Proof of Lemma \ref{mise}]
We will only prove \eqref{mise1}, as the proof of \eqref{mise3} uses similar arguments. By a standard decomposition of the weighted mean integrated square error into the sum of the weighted integrated square bias and weighted integrated variance we have
\begin{equation}
\label{mise_decomp}
\begin{aligned}
\ex\left[ \int_{\mathbb{R}} (\hat{\pi}(x)-\pi(x;\theta_0))^2 w(x) dx \right] & =  \int_{\mathbb{R}} (\ex[\hat{\pi}(x)]-\pi(x;\theta_0))^2 w(x) dx\\
&+\ex\left[\int_{\mathbb{R}} (\hat{\pi}(x)-\ex[\hat{\pi}(x)])^2 w(x) dx\right]\\
&=T_1+T_2.
\end{aligned}
\end{equation}
By assumptions of the lemma combined with Proposition 1.5 in \cite{tsybakov} it holds that
\begin{equation}
\label{t1}
T_1 \leq \| w \|_{\infty} \left( \frac{L_{\theta_0}}{\ell !} \int_{\mathbb{R}} |u|^{\alpha} |K(u)| du \right)^2 h^{2\alpha}.
\end{equation}
Next, denote
\begin{equation}
\label{yz}
Y(Z_j,x)=\frac{1}{h}K\left( \frac{x-Z_j}{h} \right)-\ex\left[ \frac{1}{h}K\left( \frac{x-Z_j}{h}\right) \right].
\end{equation}
Then
\begin{align*}
T_2 & = \frac{1}{(n+1)^2} \ex\left[ \int_{\mathbb{R}} \left( \sum_{j=0}^n Y(Z_j,x) \right)^2 w(x)dx \right]\\
& = \frac{1}{(n+1)^2} \sum_{j=0}^n \ex\left[ \int_{\mathbb{R}}    Y^2(Z_j,x)  w(x)dx \right]\\
& + \frac{1}{(n+1)^2} \sum_{i\neq j} \ex\left[ \int_{\mathbb{R}}  Y(Z_i,x)Y(Z_j,x) w(x) dx \right]\\
& = \frac{1}{n+1} \ex\left[ \int_{\mathbb{R}}    Y^2(Z_1,x)  w(x)dx \right]\\
& + \frac{1}{(n+1)^2} \sum_{i\neq j} \ex\left[ \int_{\mathbb{R}}  Y(Z_i,x)Y(Z_j,x) w(x) dx \right]\\
& = T_3+T_4
\end{align*}
holds. By Proposition 1.4 in \cite{tsybakov} we have
\begin{equation}
\label{t3}
\begin{aligned}
T_3 \leq \frac{1}{(n+1)h} \|w\|_{\infty} \int_{\mathbb{R}} K^2(u) du.
\end{aligned}
\end{equation}
Now note that
\begin{equation*}
\|Y(\cdot,\cdot)\|_{\infty} \leq 2 \|K\|_{\infty} \frac{1}{h}.
\end{equation*}
Consequently, by Lemma 3 on p.\ 10 in \cite{doukhan},
\begin{equation*}
|\ex[Y(Z_i,x)Y(Z_j,x)]| \leq 16 \|K\|_{\infty}^2 \frac{1}{h^2} \alpha_{\Delta}(|i-j|).
\end{equation*}
Thus
\begin{equation}
\label{t4}
\begin{aligned}
T_4 & \leq \frac{1}{(n+1)^2 h^2} 16 \|K\|_{\infty}^2 \| w \|_1 \sum_{i \neq j } \alpha_{\Delta}(|i-j|)\\
& = \frac{1}{(n+1)^2 h^2} 32 \|K\|_{\infty}^2 \| w \|_1 \sum_{0\leq i < j \leq n } \alpha_{\Delta}(j-i).
\end{aligned}
\end{equation}
Working out the sum on the right-hand side, we get
\begin{align*}
\sum_{0\leq i < j \leq n } \alpha_{\Delta}(j-i) & = \sum_{k=1}^{n} (n+1-k)\alpha_{\Delta}(k)\\
& \leq (n+1) \sum_{k=1}^{\infty} \alpha_{\Delta}(k),
\end{align*}
which can be seen by counting the corresponding possibilities and the trivial observation that $n+1-k\leq n+1$ for $k=1,\ldots,n.$ Note that the sum on the right-hand side of the last display is finite by Assumption \ref{cnd_mixing}. The above display, the fact that $T_2=T_3+T_4$ and the bounds \eqref{t3} and \eqref{t4} imply that
\begin{equation}
\label{t2}
T_2 \lesssim \frac{1}{nh^2}.
\end{equation}
The statement \eqref{mise1} follows from decomposition \eqref{mise_decomp} combined with formulae \eqref{t1} and \eqref{t2}. In view of the remark made at the beginning of the proof, this completes the proof of the lemma.
\end{proof}

\begin{proof}[Proof of Theorem \ref{consistency}]
We first settle the issue of measurability of $\hat{\theta}_n.$ By Lemma 2 in \cite{jennrich} to that end it is enough to have that for each fixed $\theta$ the criterion function $R_n(\theta)$ is a measurable function of the sample $Z_0,\ldots,Z_n,$ and that for $(Z_0,\ldots,Z_n)\in\mathbb{R}^{n+1}$ viewed as fixed, the function $R_n(\theta)$ is continuous in $\theta.$ However, measurability follows easily from our assumptions, while continuity of $R_n(\theta)$ in $\theta$ is a consequence of the fact that under our conditions by the corollary on p.\ 74 in \cite{whittaker} and by de la Vall\'ee Poussin's test on p.\ 72 there the function $R_n(\theta)$ is in fact three times differentiable with respect to $\theta$ (this follows by a tedious but easy verification of the assumptions made in the corollary on p.\ 74 in \cite{whittaker}). Thus in the convergence considerations we do not need to appeal to outer probability.

We will prove that
\begin{equation}
\label{convergence}
\sup_{\theta\in\Theta} |R_n(\theta)-R(\theta)|\convp 0.
\end{equation}
The statement of the theorem will then follow from this fact and assumption \eqref{identifiability} by Theorem 5.7 in \cite{vdvaart} (the fact that Chapter 5 in \cite{vdvaart} largely deals with the i.i.d.\ setting is immaterial in this case).

By the Cauchy-Schwarz inequality we have
\begin{multline*}
|R_n(\theta)-R(\theta)|\\
=\Biggl|\int_{\mathbb{R}} \left( \mu(x;\theta) \hat{\pi}(x) - \frac{1}{2} \left[ \sigma^2(x) \hat{\pi}(x) \right]^{\prime} - \mu(x;\theta)\pi(x;\theta_0) + \frac{1}{2} \left[ \sigma^2(x)\pi(x;\theta_0) \right]^{\prime} \right)\\
\times
\left( \mu(x;\theta)\hat{\pi}(x) - \frac{1}{2}\left[ \sigma^2(x) \hat{\pi}(x)\right]^{\prime} + \mu(x;\theta)\pi(x;\theta_0) - \frac{1}{2}\left[ \sigma^2(x) \hat{\pi}(x) \right]^{\prime} \right)\\
\times w(x)dx\Biggr|\\
\leq \left\{ \int_{\mathbb{R}} \left( \mu(x;\theta) (\hat{\pi}(x)-\pi(x;\theta_0)) - \frac{1}{2} \left[ \sigma^2(x)(\hat{\pi}(x)-\pi(x;\theta_0)) \right]^{\prime} \right)^2 w(x)dx \right\}^{1/2}\\
\times
\left\{  \int_{\mathbb{R}} \left( \mu(x;\theta) (\hat{\pi}(x)+\pi(x;\theta_0)) - \frac{1}{2} \left[ \sigma^2(x)(\hat{\pi}(x)+\pi(x;\theta_0)) \right]^{\prime} \right)^2 w(x)dx \right\}^{1/2}\\
=\sqrt{T_1(\theta)}\sqrt{T_2(\theta)}
\end{multline*}
with obvious definitions of $T_1(\theta)$ and $T_2(\theta).$ This inequality and Lemma \ref{lemma_t1_t2} from Appendix \ref{appendix} then yield \eqref{convergence}, which in view of the remarks we made at the beginning of this proof completes the proof of the theorem.
\end{proof}

\begin{proof}[Proof of Theorem \ref{rootn}]
Introduce the set
\begin{equation}
\label{gnepsilon}
G_{n,\varepsilon}=\{ |\hat{\theta}_n-\theta_0| \leq \varepsilon \},
\end{equation}
where $\varepsilon>0$ is some fixed number. Since $\theta_0$ is an interior point of $\Theta,$ by choosing $\varepsilon$ small enough one can achieve that on the set $G_n$ the estimator $\hat{\theta}_n$ belongs to the interior of $\Theta$ too. By the fact that $\hat{\theta}_n$ is a point of minimum of $R_n(\theta)$ it then follows that $1_{G_{n,\varepsilon}} \dot{R}_n(\hat{\theta}_n)=0.$ From this and from the mean-value theorem we have
\begin{align*}
1_{G_{n,\varepsilon}}\dot{R}_n(\theta_0) & = 1_{G_{n,\varepsilon}} \left(\dot{R}_n(\theta_0) - \dot{R}_n(\hat{\theta}_n) \right)\\
&=1_{G_{n,\varepsilon}}\int_0^1 \ddot{R}_n (\hat{\theta}_n+\lambda(\theta_0-\hat{\theta}_n))d\lambda(\theta_0-\hat{\theta}_n).
\end{align*}
The statement of the theorem follows by multiplication of the leftmost and rightmost terms of the above equality with $\sqrt{n}$ and application of Lemmas \ref{lemma_rootn1} and \ref{lemma_rootn2} from Appendix \ref{appendix}.
\end{proof}

\begin{proof}[Proof of Theorem \ref{thm_an}]
From the proofs of Theorem \ref{rootn} and Lemmas \ref{lemma_rootn1}--\ref{lemma_rootn2} from Appendix \ref{appendix} (note that our assumptions on $h$ and $\gamma$ are also used here), as well as Slutsky's lemma (Lemma 2.8 in \cite{vdvaart}) it follows that in order to establish the theorem, it is sufficient to establish asymptotic normality of
\begin{align*}
\sqrt{n+1}\int_{\mathbb{R}}v(x)(\hat{\pi}_n(x)-\pi(x;\theta_0))dx&=
\sqrt{n+1}\int_{\mathbb{R}}v(x)(\ex[\hat{\pi}_n(x)]-\pi(x;\theta_0))dx\\
&+\sqrt{n+1}\int_{\mathbb{R}}v(x)(\hat{\pi}_n(x)-\ex[\hat{\pi}_n(x)])dx.
\end{align*}
By a standard argument, cf.\ the proof of Proposition 1.2 in \cite{tsybakov}, and by our assumption on $h,$ the first term on the right-hand side of the above display converges to zero. As far as the second term is concerned, by a change of the integration variable to $u=(x-Z_j)/h$ and a simple rearrangement of the terms it can be rewritten as
\begin{multline*}
\frac{1}{\sqrt{n+1}}\sum_{j=0}^n \left\{ v(Z_j) - \ex[v(Z_j)] \right\}\\
+\frac{1}{\sqrt{n+1}}\sum_{j=0}^n \int_{-1}^{1} \{ v(Z_j+hu) - v(Z_j) \}K(u)du\\
-\sqrt{n+1} \ex\left[ \int_{-1}^{1} \{ v(Z_j+hu) - v(Z_j) \}K(u)du\right].
\end{multline*}
We want to show that the last two terms on the right-hand side of the above display vanish in probability as $n\rightarrow\infty.$ By Chebyshev's inequality it is sufficient to prove that
\begin{equation*}
\sqrt{n+1} \ex\left[ \int_{-1}^{1} \{ v(Z_j+hu) - v(Z_j) \}K(u)du\right]=o(1).
\end{equation*}
This, however, can be done through a standard argument (cf.\ the proof of Proposition 1.2 in \cite{tsybakov}) by expanding $v(Z_j+hu)$ into the Taylor polynomial of order $\alpha$ and next using the fact that $K$ is a kernel of order $\alpha,$ which yields that the left-hand side of the above display is of order $n^{1/2}h^{\alpha}=o(1).$ On the other hand, by Theorem 18.5.3 in \cite{ibragimov},
\begin{multline*}
\left\{(n+1)\left[\var[v(Z_0)]+2\sum_{j=1}^{\infty}\cov(v(Z_0),v(Z_j))\right]\right\}^{-1/2}\\
\times\sum_{j=0}^n \left\{ v(Z_j) - \ex[v(Z_j)] \right\}\convd \mathcal{N}\left(0,1\right).
\end{multline*}
Combination of the above results and Slutsky's lemma yield the statement of the theorem.
\end{proof}

\appendix
\section{  }
\label{appendix}

The present appendix contains a number of technical results used in the proofs of the main results of the paper in Section \ref{results}.

\begin{lemma}
\label{lemma_t1_t2}
Under the conditions of Theorem \ref{consistency} we have
\begin{equation}
\label{t1_bound}
\sup_{\theta\in\Theta}T_1(\theta)=o_{\mathbb{P}}(1)
\end{equation}
and
\begin{equation}
\label{t2_bound}
\sup_{\theta\in\Theta}T_2(\theta)=O_{\mathbb{P}}(1),
\end{equation}
where $T_1(\theta)$ and $T_2(\theta)$ are the same as in the proof of Theorem \ref{consistency}.
\end{lemma}
\begin{proof}
We will only prove \eqref{t1_bound}, because \eqref{t2_bound} can be proved by similar arguments. By the $c_2$-inequality and Assumption \ref{cnd_drift_volatility} we have
\begin{multline*}
\sup_{\theta\in\Theta} T_1(\theta)\lesssim \int_{\mathbb{R}} ( \hat{\pi}(x) - \pi(x;\theta_0) )^2 \widetilde{\mu}_1^2(x) w(x) dx \\
+ \int_{\mathbb{R}} \left(  \left[ \sigma^2(x) (\hat{\pi}(x)-\pi(x;\theta_0)) \right]^{\prime} \right)^2w(x)dx.
\end{multline*}
A slight variation of Lemma \ref{mise} (with a suitable choice of the weight function $w(\cdot)$ there) then shows that the right-hand side converges to zero in probability. This completes the proof of the lemma.
\end{proof}

\begin{lemma}
\label{lemma_rootn1}
Under the conditions of Theorem \ref{rootn} we have
\begin{equation*}
1_{G_{n,\varepsilon}}\sqrt{n}\dot{R}_n(\theta_0)=O_{\mathbb{P}}(1),
\end{equation*}
where the set $G_{n,\varepsilon}$ is defined in \eqref{gnepsilon}.
\end{lemma}

\begin{proof}
Differentiating under the integral sign with respect to $\theta$ the function $R_n(\theta),$ we obtain
\begin{multline*}
1_{G_{n,\varepsilon}}\sqrt{n}\dot{R}_n(\theta_0)\\
=1_{G_{n,\varepsilon}}\sqrt{n} 2 \int_{\mathbb{R}}\left( \mu(x;\theta_0)\hat{\pi}(x) - \frac{1}{2}\left[ \sigma^2(x)\hat{\pi}(x) \right]^{\prime} \right)\dot{\mu}(x;\theta_0)\hat{\pi}(x)w(x)dx.
\end{multline*}
In view of Assumption \ref{assump_characterisation} the right-hand side can be rewritten as
\begin{multline*}
1_{G_{n,\varepsilon}} 2 \sqrt{n} \int_{\mathbb{R}} \dot{\mu}(x;\theta_0)\hat{\pi}(x)w(x) \Bigl( \mu(x;\theta_0) (\hat{\pi}(x) - \pi(x;\theta_0) ) \\
-\frac{1}{2}\left[ \sigma^2(x) (\hat{\pi}(x)-\pi(x;\theta_0)) \right]^{\prime}\Bigr)dx\\
=1_{G_{n,\varepsilon}} 2 \sqrt{n} \int_{\mathbb{R}} \dot{\mu}(x;\theta_0)\pi(x;\theta_0)w(x) \mu(x;\theta_0) (\hat{\pi}(x) - \pi(x;\theta_0) ) dx \\
- 1_{G_{n,\varepsilon}} \sqrt{n} \int_{\mathbb{R}} \dot{\mu}(x;\theta_0) \pi(x;\theta_0) w(x) \left[ \sigma^2(x) (\hat{\pi}(x)-\pi(x;\theta_0)) \right]^{\prime}dx\\
+1_{G_{n,\varepsilon}} 2 \sqrt{n} \int_{\mathbb{R}} \dot{\mu}(x;\theta_0) \mu(x;\theta_0) w(x) (\hat{\pi}(x) - \pi(x;\theta_0) )^2 dx\\
-1_{G_{n,\varepsilon}} \sqrt{n} \int_{\mathbb{R}} \dot{\mu}(x;\theta_0) w(x) (\hat{\pi}(x) - \pi(x;\theta_0) ) \left[ \sigma^2(x) (\hat{\pi}(x)-\pi(x;\theta_0)) \right]^{\prime} dx\\
=T_1+T_2+T_3+T_4.
\end{multline*}
By Lemma \ref{t1t2t3t4} the terms $T_1,T_2,T_3$ and $T_4$ are $O_{\mathbb{P}}(1).$ This completes the proof.
\end{proof}

\begin{lemma}
\label{t1t2t3t4}
Let $T_1,$ $T_2,$ $T_3$ and $T_4$ be defined as in the proof of Lemma \ref{lemma_rootn1}. Then each of them is $O_{\mathbb{P}}(1).$
\end{lemma}

\begin{proof}
We start by proving the statement of the lemma for $T_1.$ We have
\begin{multline}
\label{decomp}
\int_{\mathbb{R}} \dot{\mu}(\theta_0,x)\pi(x;\theta_0)w(x) \mu(x;\theta_0) (\hat{\pi}(x) - \pi(x;\theta_0) ) dx\\
 = \int_{\mathbb{R}} \dot{\mu}(x;\theta_0)\pi(x;\theta_0)w(x) \mu(x;\theta_0) ( \ex[ \hat{\pi}(x)] - \pi(x;\theta_0) ) dx\\
 + \int_{\mathbb{R}} \dot{\mu}(x;\theta_0)\pi(x;\theta_0)w(x) \mu(x;\theta_0) (\hat{\pi}(x) - \ex[ \hat{\pi}(x)] ) dx.
\end{multline}
By Proposition 1.2 in \cite{tsybakov} it holds that
\begin{equation}
\label{bias}
\left|\int_{\mathbb{R}} \dot{\mu}(x;\theta_0)\pi(x;\theta_0)w(x) \mu(x;\theta_0) ( \ex[ \hat{\pi}(x)] - \pi(x;\theta_0) ) dx \right| \lesssim h^{\alpha} \lesssim n^{-1/2},
\end{equation}
where the last inequality follows from our assumption $h \asymp n^{-1/(2\alpha)}.$ Next we will show that the second term on the right-hand side of \eqref{decomp} is $O_{\mathbb{P}}(n^{-1/2}).$ To that end it suffices to show that
\begin{equation}
\label{variance}
\sqrt{n}\int_{\mathbb{R}}  (\hat{\pi}(x) - \ex[ \hat{\pi}(x)] ) v(x) dx = O_{\mathbb{P}}(1)
\end{equation}
for a function $v$ such that $\| v \|_{\infty}<\infty,$ because by Assumptions \ref{assump_sol} and \ref{cnd_drift_volatility}
\begin{equation*}
\| \dot{\mu}(\cdot;\theta_0)\pi(\cdot;\theta_0)w(\cdot) \mu(\cdot;\theta_0) \|_{\infty}<\infty.
\end{equation*}
By Chebyshev's inequality, the fact that $Z_j$'s are identically distributed and the fact that $\ex[Y(Z_j,x)]=0,$ where $Y(Z_j,x)$ is defined in \eqref{yz}, for an arbitrary constant $C$ we have
\begin{multline}
\label{prob_n}
P\left( \sqrt{n} \int_{\mathbb{R}}  (\hat{\pi}(x) - \ex[ \hat{\pi}(x)] ) v(x) dx   > C \right)\\
\leq \frac{n}{C^2} \var \left[ \int_{\mathbb{R}}  (\hat{\pi}(x) - \ex[ \hat{\pi}(x)] ) v(x) dx \right]\\
< \frac{1}{C^2} \frac{1}{n+1} \var\left[ \sum_{j=0}^n \int_{\mathbb{R}} Y(Z_j,x) v(x) dx \right]\\
= \frac{1}{C^2} \var\left[ \int_{\mathbb{R}} Y(Z_i,x) v(x) dx \right]\\
+ \frac{2}{C^2} \frac{1}{n+1} \sum_{0\leq i<j\leq n} \ex \left[ \int_{\mathbb{R}} Y(Z_i,x) v(x) dx \int_{\mathbb{R}} Y(Z_j,x) v(x) dx \right] .
\end{multline}
By a change of the integration variable it can be shown that
\begin{equation}
\label{y_bound}
\left| \int_{\mathbb{R}} Y(Z_i,x) v(x) dx \right| \leq 2 \| v \|_{\infty} \|K\|_1,
\end{equation}
which implies that
\begin{equation}
\label{varb}
\frac{1}{C^2} \var\left[ \int_{\mathbb{R}} Y(Z_i,x) v(x) dx \right] \leq \frac{ 4 \| v \|_{\infty}^2 \|K\|_1^2}{ C^2 }.
\end{equation}
Furthermore, using \eqref{y_bound} we get for $i<j$ from Lemma 3 on p.\ 10 in \cite{doukhan} that
\begin{equation*}
\left| \ex\left[ \int_{\mathbb{R}} Y(Z_i,x) v(x) dx \int_{\mathbb{R}} Y(Z_j,x) v(x) dx \right] \right| \leq 16 \| v \|_{\infty}^2 \|K\|_1^2 \alpha_{\Delta}(j-i).
\end{equation*}
By counting the cases when $j-i=k$ for $k=1,\ldots,n,$ it can be seen that
\begin{multline*}
\left|  \frac{2}{C^2} \frac{1}{n+1} \sum_{0\leq i<j\leq n} \ex \left[ \int_{\mathbb{R}} Y(Z_i,x) v(x) dx \int_{\mathbb{R}} Y(Z_j,x) v(x) dx \right] \right| \\
\leq \frac{32}{C^2} \frac{1}{n+1}  \| v \|_{\infty}^2 \|K\|_1^2 \sum_{k=1}^{n} (n+1-k)\alpha_{\Delta}(k)\\
\leq \frac{1}{C^2}  32 \| v \|_{\infty}^2 \| K \|_1^2 \sum_{k=1}^{\infty}\alpha_{\Delta}(k).
\end{multline*}
The finiteness of the sum in the rightmost term of the above display is guaranteed by Assumption \ref{cnd_mixing}. The above display and \eqref{varb} show that the left-hand side of \eqref{prob_n} can be made arbitrarily small by selecting $C$ large, which shows that \eqref{variance} holds. Formulae \eqref{decomp}--\eqref{variance} then imply that $T_1$ is $O_{\mathbb{P}}(1).$

Next we treat $T_2.$ By integration by parts and using Assumption \ref{cnd_drift_volatility},
\begin{equation*}
T_2 = 1_{G_{n,\varepsilon}} \sqrt{n} \int_{\mathbb{R}}  \left[ \dot{\mu}(x;\theta_0) \pi(x;\theta_0) w(x) \right]^{\prime} \sigma^2(x) (\hat{\pi}(x)-\pi(x;\theta_0)) dx.
\end{equation*}
The right-hand side can be treated by exactly the same arguments as used above for $T_1$ and one can show that $T_2=O_{\mathbb{P}}(1).$

We move to $T_3.$ By Chebyshev's inequality
\begin{multline*}
P\left( \sqrt{n} \int_{\mathbb{R}} \dot{\mu}(x;\theta_0) \mu(\theta_0,x) w(x) (\hat{\pi}(x) - \pi(x;\theta_0) )^2 dx > C \right)\\
\leq \frac{1}{C} \sqrt{n} \ex\left[ \int_{\mathbb{R}} | \mu(x;\theta_0) \dot{\mu}(x;\theta_0) | w(x)  (\hat{\pi}(x) - \pi(\theta_0,x) )^2 dx \right].
\end{multline*}
By a slight variation of the statement of Lemma \ref{mise} (replace $w(\cdot)$ in the statement with $\widetilde{\mu}_2(\cdot)\widetilde{\mu}_1(\cdot)w(\cdot)$) the right-hand side of the above display is $o(1)$ and hence $T_3$ is $o_{\mathbb{P}}(1).$

Finally, $T_4$ can be handled by the same argument as $T_3$ employing the Cauchy-Schwarz inequality to see that
\begin{multline*}
\left| \int_{\mathbb{R}} \dot{\mu}(x;\theta_0) w(x) (\hat{\pi}(x) - \pi(x;\theta_0) ) \left[ \sigma^2(x) (\hat{\pi}(x)-\pi(x;\theta_0)) \right]^{\prime} dx \right|\\
\leq
\left\{ \int_{\mathbb{R}} ( \dot{\mu}^2(x;\theta_0) )^2w(x) ( \hat{\pi}(x)-\pi(x;\theta_0) )^2 dx \right\}^{1/2} \\
\times \left\{ \int_{\mathbb{R}} w(x) \left( \left[ \sigma^2(x) (\hat{\pi}(x)-\pi(x;\theta_0)) \right]^{\prime} \right)^2 dx \right\}^{1/2}.
\end{multline*}
Next the arguments similar to those given above together with Lemma \ref{mise} allow one to conclude that the right-hand side is $O_{\mathbb{P}}(n^{-1/2})$ and hence $T_4=O_{\mathbb{P}}(1).$ This completes the proof of the lemma.
\end{proof}

\begin{lemma}
\label{lemma_rootn2}
Under the conditions of Theorem \ref{rootn} we have
\begin{equation*}
1_{G_{n,\varepsilon}}\int_0^1 \ddot{R}_n (\hat{\theta}_n+\lambda(\theta_0-\hat{\theta}_n))d\lambda\convp \ddot{R} (\theta_0),
\end{equation*}
where the set $G_{n,\varepsilon}$ is defined in \eqref{gnepsilon}.
\end{lemma}
\begin{proof}
We have
\begin{multline*}
1_{G_{n,\varepsilon}} \int_0^1 \ddot{R}_n (\hat{\theta}_n+\lambda(\theta_0-\hat{\theta}_n))d\lambda
=1_{G_{n,\varepsilon}}  \ddot{R}_n (\theta_0) \\
+1_{G_{n,\varepsilon}} \int_0^1 \left( \ddot{R}_n (\hat{\theta}_n+\lambda(\theta_0-\hat{\theta}_n)) - \ddot{R}_n (\theta_0) \right)d\lambda=T_1+T_2.
\end{multline*}
By Lemma \ref{term1} the term $T_1$ converges in probability to $\ddot{R}(\theta_0),$ while by Lemma \ref{term2} the term $T_2$ converges in probability to zero. This completes the proof.
\end{proof}

\begin{lemma}
\label{term1}
For $T_1$ defined as in the proof of Lemma \ref{lemma_rootn2} and under the same conditions as in Lemma \ref{lemma_rootn2} we have
$T_1\convp \ddot{R}(\theta_0).$
\end{lemma}
\begin{proof}
By consistency of $\hat{\theta}_n,$ see Theorem \ref{consistency}, we have $1_{G_{n,\varepsilon}}\convp 1.$ Furthermore,
\begin{multline}
\label{threeA}
\ddot{R}_n(\theta_0)=2\int_{\mathbb{R}}\dot{\mu}^2(x;\theta_0) \hat{\pi}^2(x) w(x)dx\\
+2\int_{\mathbb{R}} \ddot{\mu}(x;\theta_0) \mu(x;\theta_0) \hat{\pi}^2(x) w(x)dx\\
-\int_{\mathbb{R}}\left[ \sigma^2(x) \hat{\pi}(x) \right]^{\prime}\ddot{\mu}(x;\theta_0)\hat{\pi}(x)w(x)dx\\
=A_1+A_2+A_3.
\end{multline}
We will treat each of the three terms on the right-hand side separately. First of all,
\begin{multline*}
A_1=2\int_{\mathbb{R}}\dot{\mu}^2(x;\theta_0) \pi^2(x;\theta_0)w(x)dx\\
+2\int_{\mathbb{R}}\dot{\mu}^2(x;\theta_0) \left\{ \hat{\pi}^2(x)  - \pi^2(x;\theta_0) \right\}w(x)dx=A_4+A_5.
\end{multline*}
We will show that $A_5$ is $o_{\mathbb{P}}(1).$ By the Cauchy-Schwarz inequality combined with the $c_2$-inequality we have
\begin{multline*}
|A_5|\leq 2 \left\{ \int_{\mathbb{R}}  \dot{\mu}^2(x;\theta_0)  ( \hat{\pi}(x) - \pi(x;\theta_0) )^2w(x)dx\right\}^{1/2}\\
\times \Biggl\{ 2 \int_{\mathbb{R}} \dot{\mu}^2(x;\theta_0)  ( \hat{\pi}(x) - \pi(x;\theta_0) )^2w(x)dx\\ + 8 \int_{\mathbb{R}} \dot{\mu}^2(x;\theta_0)  \pi^2(x;\theta_0) w(x)dx \Biggr\}^{1/2}.
\end{multline*}
The right-hand side is $o_{\mathbb{P}}(1)$ by Lemma \ref{mise}, and hence so is $A_5.$ Thus
\begin{equation}
\label{A1}
A_1=A_4+o_{\mathbb{P}}(1)=2\int_{\mathbb{R}}\dot{\mu}^2(x;\theta_0)  \pi^2(x;\theta_0) w(x)dx+o_{\mathbb{P}}(1).
\end{equation}
Now we turn to $A_2.$ By the same reasoning as used for $A_1,$ one can show that
\begin{equation}
\label{A2}
A_2=2\int_{\mathbb{R}} \ddot{\mu}(x;\theta_0) \mu(x;\theta_0) \pi^2(x;\theta_0)  w(x)dx+o_{\mathbb{P}}(1).
\end{equation}
Finally, a long and tedious computation, which is omitted to save the space and which is similar to the one used to study $A_1,$ shows that
\begin{equation}
\label{A3}
A_3=\int_{\mathbb{R}} \left[ \sigma^2(x) \pi(x;\theta_0) \right]^{\prime}\ddot{\mu}(x;\theta_0)\pi(x;\theta_0)w(x)dx+o_{\mathbb{P}}(1).
\end{equation}
The statement of the lemma follows upon collecting formulae \eqref{A1}--\eqref{A3} and using the representation \eqref{threeA}.
\end{proof}
\begin{lemma}
\label{term2}
For $T_2$ defined as in the proof of Lemma \ref{lemma_rootn2} and under the same conditions as in Lemma \ref{lemma_rootn2} we have
$T_2\convp 0.$
\end{lemma}
\begin{proof}
Denote $\Phi_n(\theta)=\ddot{R}_n(\theta).$ Using the mean-value theorem, we have the following chain of inequalities,
\begin{multline*}
\left| 1_{G_{n,\varepsilon}} \int_0^1 ( \Phi_n(\hat{\theta}_n+\lambda(\theta_0-\hat{\theta}_n))-\Phi_n(\theta_0) )d\lambda \right|\\
=  1_{G_{n,\varepsilon}}\left|  \int_0^1 (1-\lambda) d\lambda \int_0^1  \dot{\Phi}_n ( \theta_0 + \psi(1-\lambda)(\hat{\theta}_n-\theta_0) )d\psi(\hat{\theta}_n-\theta_0)  \right| \\
\leq 1_{G_{n,\varepsilon}} \int_0^1 d\lambda \int_0^1  \left| \dot{\Phi}_n ( \theta_0 + \psi(1-\lambda)(\hat{\theta}_n-\theta_0) )\right|d\psi|\hat{\theta}_n-\theta_0|.
\end{multline*}
Since $|\hat{\theta}_n-\theta_0|=o_{\mathbb{P}}(1)$ by Theorem \ref{consistency}, in order to prove the lemma it suffices to show that
\begin{equation}
\label{multf}
1_{G_{n,\varepsilon}} \int_0^1 d\lambda \int_0^1 \left| \dot{\Phi}_n ( \theta_0 + \psi(1-\lambda)(\hat{\theta}_n-\theta_0) )\right|d\psi=O_{\mathbb{P}}(1).
\end{equation}
Observe that
\begin{align*}
\dot{\Phi}_n(\theta)&=\dddot{R}_n(\theta)\\
&=4\int_{\mathbb{R}} \dot{\mu}(x;\theta)\ddot{\mu}(x;\theta) \hat{\pi}^2(x)w(x)dx\\
&+2\int_{\mathbb{R}} \dddot{\mu}(x;\theta) \mu(x;\theta)\hat{\pi}^2(x)w(x)dx\\
&+2\int_{\mathbb{R}} \ddot{\mu}(x;\theta) \dot{\mu}(x;\theta)\hat{\pi}^2(x)w(x)dx\\
&-\int_{\mathbb{R}} \left[ \sigma^2(x) \hat{\pi}(x) \right]^{\prime} \dddot{\mu}(x;\theta) \hat{\pi}(x)w(x)dx\\
&=A_1(\theta)+A_2(\theta)+A_3(\theta)+A_4(\theta),
\end{align*}
where differentiation under the integral sign is justified by the corollary on p.\ 72 in \cite{whittaker}, by de la Vall\'ee Poussin's test on p.\ 72 there and by our assumptions. Next insert the expression above into the left-hand side of formula \eqref{multf}. Denoting
\begin{equation*}
\hat{\theta}_{n,\psi,\lambda}=\theta_0+\psi(1-\lambda)(\hat{\theta}_n-\theta_0),
\end{equation*}
we see that we need to show that
\begin{equation*}
1_{G_{n,\varepsilon}} \int_0^1 d\lambda \int_0^1 \left| \sum_{i=1}^4 A_i (\hat{\theta}_{n,\psi,\lambda}) \right| d\psi=O_{\mathbb{P}}(1).
\end{equation*}
By appropriately selecting $\varepsilon$ in the definition of the set $G_{n,\varepsilon}$ in \eqref{gnepsilon}, one can achieve that for all $\lambda,\psi\in[0,1]$ one has that $\hat{\theta}_{n,\psi,\lambda}$ belongs to the interior of the parameter set $\Theta.$ Keeping this in mind, we need to study the term
\begin{equation}
\label{A1OP1}
1_{G_{n,\varepsilon}} \int_0^1 d\lambda \int_0^1 \left| A_i (\hat{\theta}_{n,\psi,\lambda}) \right| d\psi
\end{equation}
for $i=1.$ The arguments for other terms with $i=2,3,4$ are similar and are omitted. We have
\begin{multline*}
1_{G_{n,\varepsilon}} \int_0^1 d\lambda \int_0^1 \left| A_1 (\hat{\theta}_{n,\psi,\lambda}) \right| d\psi \\
\leq 8 \int_{\mathbb{R}} \widetilde{\mu}_2(x)\widetilde{\mu}_3(x)( \hat{\pi}(x) - \pi(x;\theta_0) )^2w(x)dx\\
+8 \int_{\mathbb{R}} \widetilde{\mu}_2(x)\widetilde{\mu}_3(x) \pi^2(x;\theta_0) w(x)dx,
\end{multline*}
from which and from Lemma \ref{mise} it is immediate that \eqref{A1OP1} is $O_{\mathbb{P}}(1)$ for $i=1.$ In the light of the remarks made above this completes the proof.
\end{proof}

\bibliographystyle{plainnat}

\end{document}